\documentclass[12pt]{article}
\usepackage[utf8]{inputenc}
\usepackage{amsmath}
\usepackage{amssymb}
\usepackage{amsthm}
\usepackage{fullpage}
\usepackage{graphicx}
\usepackage{hyperref}
\usepackage{color}
\usepackage[sc]{mathpazo}
\usepackage[T1]{fontenc}
\usepackage[margin=10pt,font=small]{caption}

\newtheorem{thm}{Theorem}[section]
\newtheorem{prop}[thm]{Proposition}

\newtheorem{coro}[thm]{Corollary}

\theoremstyle{remark}

\newcommand{\tdef}[1]{\textcolor{blue}{\emph{#1}}}
\newcommand{\horiz}{\begin{center}\rule{0.3\textwidth}{0.5pt}\end{center}}

\newcommand{\motz}{\mathcal{M}}
\newcommand{\schr}{\mathcal{S}}
\newcommand{\dyck}{\mathcal{D}}
\newcommand{\type}{\textrm{Type}}
\newcommand{\cls}{\textrm{cls}}
\newcommand{\covered}{\lessdot}
\newcommand{\cont}{\textrm{cont}}
\newcommand{\ds}{\textrm{ds}}

\newcommand{\OEIS}[1]{\href{http://oeis.org/#1}{#1}}

\author{Wenjie Fang \\ Université Paris-Est Marne-la-Vallée, LIGM (UMR 8094), \\ CNRS, ENPC, ESIEE Paris, France\thanks{Wenjie Fang was supported by Austria FWF Grant I2309-N35 and P27290 during the conduction of this work.}}

\title{A partial order on Motzkin paths}
\begin{document}

\maketitle

\abstract{The Tamari lattice, defined on Catalan objects such as binary trees and Dyck paths, is a well-studied poset in combinatorics. It is thus natural to try to extend it to other families of lattice paths. In this article, we fathom such a possibility by defining and studying an analogy of the Tamari lattice on Motzkin paths. While our generalization is not a lattice, each of its connected components is isomorphic to an interval in the classical Tamari lattice. With this structural result, we proceed to the enumeration of components and intervals in the poset of Motzkin paths we defined. We also extend the structural and enumerative results to Schr\"oder paths. We conclude by a discussion on the relation between our work and that of Baril and Pallo (2014).}

\horiz

\section{Introduction}

The Tamari lattice is a poset defined on Catalan objects such as Dyck paths and binary trees. First proposed by Tamari \cite{tamari}, it is a well-studied object in combinatorics, and is also the basis of many other objects, such as the associahedron \cite{associahedron} and the Loday-Ronco Hopf algebra \cite{loday1998hopf}. It also has several generalizations, such as the $m$-Tamari lattice \cite{bergeron-preville} and the generalized Tamari lattice \cite{PRV2014extension}. Recently, there is a trend on the enumerative and bijective study of intervals in the Tamari lattice \cite{chapoton-tamari,BB2009intervals,interval-poset,sticky}, from which we can see the rich combinatorics there to be mined.

Since the Tamari lattice can be defined on Dyck paths (see Proposition~2.1 in \cite{BB2009intervals}), it is natural to ask for its extension to other types of lattice paths. In this article, we take the first step in this direction by defining a partial order on Motzkin paths, a family of lattice paths not far away from Dyck paths, using rules similar to that of the Tamari lattice. We find that the poset of Motzkin paths of length $n$ defined in this way is not connected, therefore not a lattice in general. However, there is a bijection of Callan \cite{callan} from Motzkin paths to a certain family of Dyck paths that preserves the order structure. With this bijection, we prove that each connected component of the poset of Motzkin paths is isomorphic to a certain generalized Tamari lattice, which is in turn isomorphic to an interval in the classical Tamari lattice. We then study the enumerative aspects of the poset of Motzkin paths, such as the number of connected components and the number of intervals. We find that the generating function of intervals in the poset of Motzkin paths, weighted by the number of diagonal steps and contacts (details are postponed to later sections), is algebraic. This result is obtained by solving a functional equation ``with one catalytic variable'', as treated in \cite{BMJ}. The same study is then extended to Schr\"oder paths, where similar results are established.

There are previous efforts on defining partial orders on Motzkin paths. In \cite{MR2202340}, Ferrari and Pinzani constructed partial orders of different families of lattice paths, including Motzkin paths. They also proved that, in some cases, including Dyck paths, Motzkin paths and Schr\"oder paths, the defined partial order is a distributive lattice. Their construction, which is based on weak dominance of paths, is clearly different from ours. In \cite{MR3128387}, Baril and Pallo analyzed the sub-poset of the Tamari lattice induced by their so-called ``Motzkin words''. The result of Baril and Pallo is similar to ours in the sense that both posets can be defined on the same set of paths, but also fundamentally different in the sense that we consider different orders on these objects. The relation of \cite{MR3128387} and our work will be discussed in the last section.

This article is organized as follows. In Section~2, as preliminary, we give the definition of our poset on Motzkin paths, and some related definitions useful in later sections. Then, in Section~3, we establish some structural results on our poset of Motzkin paths. Section~4 consists of an enumerative study of our poset of Motzkin paths, including finding out the generating function of intervals in the defined poset. The whole set of results is then transferred to Schr\"oder paths in Section~5. We conclude with some remarks in Section~6.

\section{Preliminaries} \label{sec:prelim}

We consider lattices paths on $\mathbb{Z}^2$ starting at $(0,0)$, ending on the diagonal $y=x$ without crossing it, and composed by three types of steps: north step $N = (0, 1)$, east step $E = (1, 0)$ and diagonal step $D = (1, 1)$. Such a path $P$ is called a \tdef{Motzkin path}, and if $P$ consists of only north and east steps, then it is also called a \tdef{Dyck path}. It is clear that all Dyck paths are Motzkin paths. We say that a path is of size $n$ if it consists of $n$ steps. We denote by $\dyck_n$ and $\motz_n$ the set of Dyck paths and Motzkin paths of size $n$ respectively. It is clear that $\dyck_{2n+1}$ is empty for any natural number $n$. We should also note that not all Motzkin paths of the same size end at the same point. The set of all Motzkin paths (resp. Dyck paths) is denoted by $\motz$ (resp. $\dyck$). Both Motzkin paths and Dyck paths can be viewed as words in the alphabet $\{N, E, D\}$.

In the following, we will always use $P, Q$ and their variants for Motzkin paths, and $R, S$ and their variants for Dyck paths. We denote by $\epsilon$ the empty path, and we take the convention that $\epsilon$ is not counted as a Dyck path or a Motzkin path.

It is well-known that the number of Dyck paths of size $2n$ is given by the $n^{\rm th}$ Catalan number $\mathrm{Cat}_n = \frac1{2n+1} \binom{2n+1}{n}$. Motzkin paths of size $n$ are given by the so-called $n^{\rm th}$ Motzkin number, whose formula is not as nice as that of Catalan numbers. The first few Motzkin numbers (index starting at $1$) are
\[ 1, 2, 4, 9, 21, 51, 127, 323, \ldots \]
This is the sequence \OEIS{A001006} on the Online Encyclopedia of Integer Sequences (OEIS).

We now consider a poset defined on Motzkin paths of size $n$, inspired by the Tamari lattice on Dyck paths (see Proposition~2.1 of \cite{BB2009intervals}). Given a Motzkin path $P$, if a lattice point $v$ on $P$ is preceded by an east step and succeeded by a north step or a diagonal step, then $v$ is called a \tdef{valley}. We can also consider valleys as endpoints of consecutive east steps. Then, for a valley $v$ in $P$, let $w$ be the next lattice point on $P$ with the same horizontal distance to the main diagonal. We denote by $S$ the sub-path of $P$ between $v$ and $w$. Since $v$ is a valley, $S$ is preceded by an east step. By exchanging $S$ with the preceding east step, we obtain a new path $Q$, which is also a Motzkin path, and we say that $Q$ \tdef{covers} $P$, denoted by $P \lessdot_\motz Q$. Figure~\ref{fig:motz-cover} illustrates two examples of covering. Taking all possibilities of valley points $v$ in every Motzkin path $P$ of length $n$, we construct a covering relation, which is then extended by transitivity to a partial order $\leq_\motz$ on $\motz_n$. This partial order $(\leq_\motz, \motz_n)$ is our subject of study. The same procedure applied to Dyck paths of length $2n$ gives the Tamari lattice of order $n$, denoted by $(\leq_D, \dyck_{2n})$.

\begin{figure}
  \begin{center}
    \includegraphics[scale=0.8,page=1]{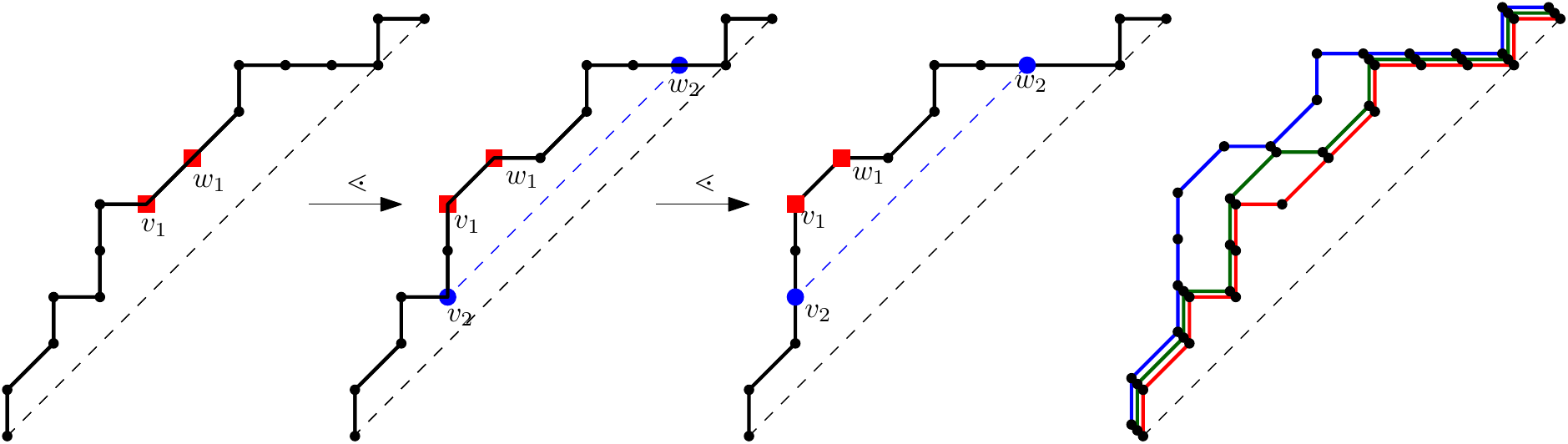}
  \end{center}
  \caption{Examples of covering relations on Motzkin paths}
  \label{fig:motz-cover}
\end{figure}

Unlike the Tamari lattice on Dyck paths, the partial order $(\leq_\motz, \motz_n)$ defined above is not a lattice. In fact, its Hasse diagram is not even connected. Figure~\ref{fig:motz-comp} illustrates some connected components of the Hasse diagram of $(\leq_\motz, \motz_6)$. To understand the structure of its connected components, we need a few definitions to distinguish different sub-classes of Motzkin paths.

\begin{figure}
  \centering
  \includegraphics[scale=0.7,page=3]{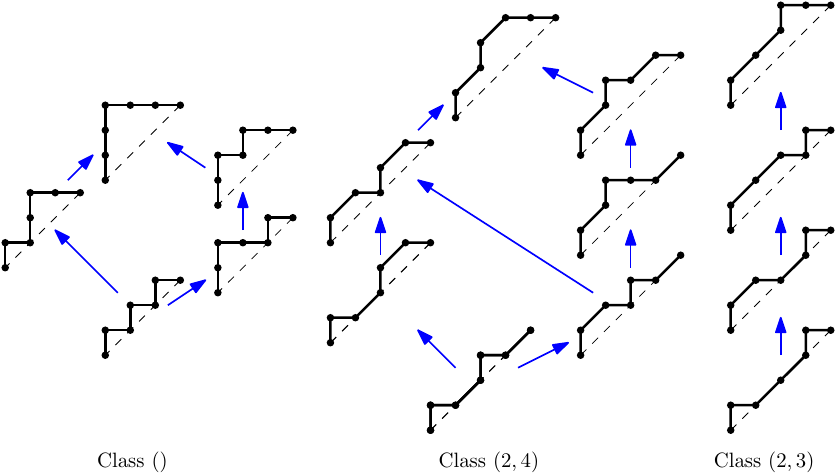}
  \caption{Some connected components of the Hass diagram of $(\leq_\motz, \motz_6)$.}
  \label{fig:motz-comp}
\end{figure}

We say that a diagonal step $D$ is of height $h$ if it ends with $y$-coordinate $h$. The \tdef{class} of a Motzkin path $P$, denoted by $\cls(P)$, is the sequence of the heights of its diagonal steps in increasing order. For instance, for the Motzkin path $P = NNDEDNNEEDE$, we have $\cls(P) = (3,4,7)$. As another example, the Motzkin paths in Figure~\ref{fig:motz-cover} are all in the class $(2,6,7)$. Equivalently, the $i$-th component of $\cls(P)$ is given by the number of north steps and diagonal steps that come before the end of the $i$-th diagonal step (including itself). The length of $\cls(P)$, which is the number of diagonal steps $D$ in $P$, is denoted by $|P|_D$. The following proposition, whose proof is straightforward from the definition of $(\leq_\motz, \motz_n)$, shows how the classes govern connected components in $(\leq_\motz, \motz_n)$.

\begin{prop} \label{prop:comp-type}
  For two Motzkin paths $P, Q$ such that $P \leq_\motz Q$, we have $\cls(P) = \cls(Q)$.
\end{prop}

This result implies that Motzkin paths in the same connected component of $(\leq_\motz, \motz_n)$ are in the same class. A natural question thus arises: do the classes characterize all connected components? In other words, given two paths $P$ and $Q$ of the same class, are they always in the same connected components? In Figure~\ref{fig:motz-comp}, the answer seems to be yes. To answer this question, and to look at the structure of all the connected components of $(\leq_\motz, \motz_n)$, we need to take a detour over Dyck paths.

\section{Motzkin paths and Dyck paths}

Following \cite{nonsep}, for a Dyck path $R \in \dyck_{2n}$ of length $2n$, we define its \tdef{type}, denoted by $\type(R)$, to be a word $w$ of length $n-1$, such that the $i$-th letter $w_i$ is $N$ if the $i$-th north step $N_i$ in $R$ is followed by an \emph{east step}, and $w_i=E$ otherwise. The notion of type for Dyck paths corresponds in fact to the canopy of a binary tree defined in \cite{PRV2014extension}, which indicates for each leaf in prefix order whether it is the left or right child of it parent. We now consider an interval $I=[R,S]$ in the Tamari lattice $(\leq_D, \dyck_{2n})$, where $R$ (resp. $S$) is its minimal (resp. maximal) element. We can thus identify intervals in the Tamari lattices with pairs of comparable elements. We say that the interval $[R,S]$ is \tdef{synchronized} if $\type(R) = \type(S)$. The left side of Figure~\ref{fig:bij-dyck-motz} is an example of a synchronized interval of type $ENNENENENNNENE$.

\begin{figure}
  \begin{center}
    \includegraphics[page=2,scale=0.8]{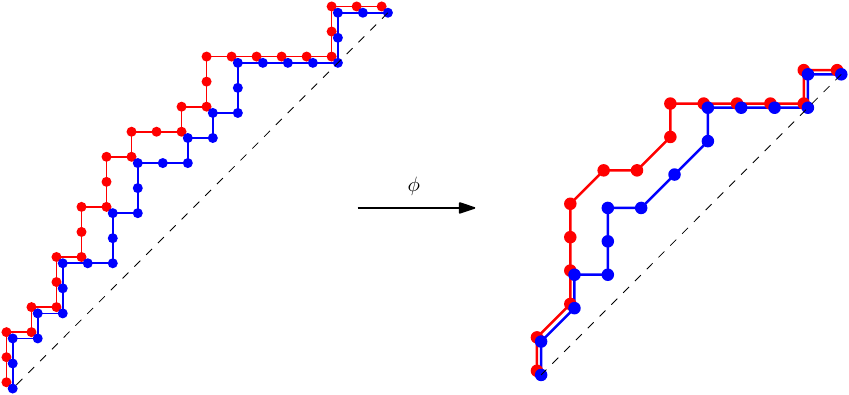}
  \end{center}
  \caption{Bijection $\phi$ on Dyck paths avoiding $NNN$}
  \label{fig:bij-dyck-motz}
\end{figure}

We now consider Dyck paths that avoid three consecutive north steps $NNN$. We denote by $\dyck^\circ$ the set of such Dyck paths. It is clear that a Dyck path $R$ is in $\dyck^\circ$ if and only if its type $\type(R)$ avoids $EE$. In \cite{callan}, Callan proposed the following bijection $\phi$ from $\dyck^\circ$ to $\motz$, which is reformulated here for our need. Given a Dyck path $R$ from $\dyck^\circ$, it takes the form
\[
R = N^{a_1}E^{b_1}N^{a_2}E^{b_2} \cdots N^{a_k}E^{b_k},
\]
where $a_i \in \{1,2\}$ and $b_i > 0$ for all indices $i$. Since $R$ avoids the pattern $NNN$ as a word, all $a_i$'s are either $1$ or $2$. We now define a function $\xi$ with $\xi(1)=D$ and $\xi(2)=N$, and we define $\phi(R)$ by
\[
\phi(R) = \xi(a_1)E^{b_1-1}\xi(a_2)E^{b_2-1} \cdots \xi(a_k)E^{b_k-1}.
\]
In other words, for each maximal sub-word of the form $N^{a_i}$, which is followed by at least one $E$, if there is only one $N$, then we replace $NE$ by $D$; if there are two $N$'s, then we replace $NNE$ by $N$. Geometrically, it is clear that $\phi(R)$ never goes beneath the main diagonal, thus is a Motzkin path. The reverse direction $\phi^{-1}$ is just replacing $D$ by $NE$ and $N$ by $NNE$ in a Motzkin path. This is clearly a bijection between $\dyck^\circ$ and $\motz$.

Two examples of $\phi$ are given in Figure~\ref{fig:bij-dyck-motz}, where the two paths of a synchronized interval avoiding $NNN$ are mapped to two Motzkin paths. In this example, we notice that the resulted Motzkin paths have the same type and are comparable in the Motzkin poset. In the rest of this section, we will prove that this phenomenon is not a coincidence.

We have the following property of $\phi$ concerning the type of a Dyck path and the class of its image of $\phi$.

\begin{prop} \label{prop:phi-type}
  Given $R, S \in \dyck^\circ_{2n}$ two Dyck paths of length $2n$ avoiding $NNN$, we have $\type(R) = \type(S)$ if and only if $\cls(\phi(R)) = \cls(\phi(S))$.
\end{prop}
\begin{proof}
  Since $R$ avoids $NNN$, it takes the form $R = N^{a_1}E^{b_1}N^{a_2}E^{b_2} \cdots N^{a_k}E^{b_k}$, with $a_i \in \{1, 2\}$ and $b_i > 0$ for all indices $i$. Therefore, the type of $R$ depends entirely on the values of all $a_i$'s. More precisely, let $\nu_t$ be the function with $\nu_t(1)=N$ and $\nu_t(2)=EN$, then we have $\nu_t(a_1) \nu_t(a_2) \cdots \nu_t(a_k) = \type(R)N$. Conversely, given $\type(R)$, we can determine the values of all $a_i$'s. We can thus say that the sequence $(a_1, \ldots, a_k)$ encodes bijectively $\type(R)$.

  Now, suppose that there are $\ell$ terms among all $a_i$'s that take the value $1$, then there are exactly $\ell$ diagonal steps in $\phi(R)$. Let $d_1, \ldots, d_\ell$ be all indices with $a_{d_i}=1$. It is clear that, given all $d_i$'s, we can recover all $a_i$'s. We now look at the class of $\phi(R)$. Suppose that $\cls(\phi(R)) = (c_1, \ldots, c_\ell)$, and we recall that the $c_i$ is the number of north and diagonal steps that comes before the $i^{\rm th}$ diagonal step, including itself, in $\phi(R)$. Therefore, $c_i = d_i$ by construction. Given $\cls(\phi(R))$, we can recover all $a_i$'s. We can thus say that $(a_1, \ldots, a_k)$ encodes bijectively $\cls(\phi(R))$.

  As the sequence $(a_1, \ldots, a_k)$ encodes bijectively both $\type(R)$ and $\cls(\phi(R))$, we conclude that $\type(R) = \type(S)$ if and only if $\cls(\phi(R)) = \cls(\phi(S))$.
\end{proof}

We denote by $\cdot$ the concatenation operator of paths. We now consider how $\phi$ interacts with both partial orders $(\leq_D, \dyck^\circ_{2n})$ and $(\leq_\motz, \motz_n)$. We say that a Dyck path (or Motzkin path) is \tdef{primitive} if it only touches the diagonal at its start and end points. It is easy to see that a primitive Dyck path takes the form $N \cdot R \cdot E$, with $R$ a Dyck path. For a primitive Motzkin path, either it takes the form $N \cdot P \cdot E$ with $P$ a Motzkin path, or it consists of one single diagonal step. We have the following property of $\phi$.

\begin{prop} \label{prop:phi-primitive}
  Let $R \in \dyck^\circ_{2n}$ be a Dyck path avoiding $NNN$, then $R$ is primitive if and only if $P = \phi(R)$ is a primitive Motzkin path.
\end{prop}
\begin{proof}
  Since $R$ can be written as $R = N^{a_1}E^{b_1}N^{a_2}E^{b_2} \cdots N^{a_k}E^{b_k}$, with $a_i \in \{1,2\}$ and $b_i > 0$, we will work on the $a_i$'s and $b_i$'s instead. By definition, $R$ is primitive if and only if $\sum_{i=1}^t (a_i - b_i) > 0$ for all $0 < t < k$. It is because we only need to check whether the path touches the diagonal at the end of consecutive east steps.

  Now, we know that $P = \phi(R) = \xi(a_1)E^{b_1-1}\xi(a_2)E^{b_2-1} \cdots \xi(a_k)E^{b_k-1}$, with $\xi(1) = D$ and $\xi(2)=N$. Using the same reasoning for $R$, we know that $P$ is primitive if and only if $\sum_{i=1}^t (\mu(a_i) - b_i + 1) > 0$ for all $0 < t < k$, with $\mu(a)$ defined by $\mu(1) = 0$, $\mu(2) = 1$. We observe that $\mu(a) = a-1$, therefore, $a_i - b_i = \mu(a_i) - b_i + 1$. We conclude by the observation that the two conditions of being primitive for $R$ and for $P$ are equivalent.
\end{proof}

The covering relations in both $(\leq_D, \dyck_{2n})$ and $(\leq_\motz, \motz_n)$ can be reformulated as follows. Given two Motzkin paths $P, Q$, we have $Q$ covers $P$ if and only if we can write $P = P_1 \cdot E \cdot P_2 \cdot P_3$ with $P_2$ a non-empty primitive Motzkin path, such that $Q = P_1 \cdot P_2 \cdot E \cdot P_3$. The condition also holds for Dyck paths. We now prove the following cornerstone result.

\begin{thm} \label{thm:phi-cover}
  Let $R, S$ be two Dyck paths avoiding $NNN$, and $P = \phi(R)$, $Q = \phi(S)$ their corresponding Motzkin paths. Then $P \leq_\motz Q$ if and only if $R \leq_D S$ and $\type(R) = \type(S)$, that is, $[R,S]$ is a synchronized interval that avoids $NNN$.
\end{thm}
\begin{proof}
  We will prove a stronger result: $P \covered_\motz Q$ if and only if $R \covered_D S$ and $\type(R) = \type(S)$, which clearly implies our claim.

  We first prove the ``only if'' part. Since $P \covered_\motz Q$, by the primitive path reformulation of $\covered_D$, we can write $P = P_1 \cdot E \cdot P_2 \cdot P_3$ such that $Q = P_1 \cdot P_2 \cdot E \cdot P_3$, with $P_2$ a non-empty primitive Motzkin path. Then, in the bijection $\phi^{-1}$, since the east steps are left untouched, we have $R = R_1 \cdot E \cdot \phi^{-1}(P_2) \cdot R_3$, where $R_1$ (resp. $R_3$) is obtained from $P_1$ (resp. $P_3$) using the same substitution as in $\phi^{-1}$. We also have $S = R_1 \cdot \phi^{-1}(P_2) \cdot E \cdot R_3$. By Proposition~\ref{prop:phi-primitive}, the Dyck path $\phi^{-1}(P_2)$ is also primitive. We thus conclude that $R \covered_\dyck S$. For the type of $R$ and $S$, we notice that the substitution in $\phi^{-1}$ transforms all possible steps $N$, $E$, $D$ into paths $NNE$, $E$, $NE$, all ending in $E$. Therefore, $R_1$ ends in $E$, meaning that swapping $\phi^{-1}(P_2)$ with $E$ in $P$ does not change the type.

  For the ``if'' part, since $R \covered_D S$, we can write $R = R_1 \cdot E \cdot R_2 \cdot R_3$ such that $S = R_1 \cdot R_2 \cdot E \cdot R_3$, with $R_2$ a non-empty primitive Dyck path, which begins with $N$ and ends with $E$. As $\type(R)=\type(S)$, the path $R_1$ must end in $E$ to avoid change of type. Since both $R_1$ and $R_2$ ends in $E$, in the substitution of $\phi$, all segments $R_1, R_2, R_3$ are independent. We thus have $P = P_1 \cdot E \cdot \phi(R_2) \cdot P_3$ and $Q = P_1 \cdot \phi(R_2) \cdot E \cdot P_3$, with $P_1$ and $P_3$ obtained respectively from $R_1$ and $R_3$ with the substitution of $\phi$. By Proposition~\ref{prop:phi-primitive}, the Motzkin path $\phi(R_2)$ is also primitive. We thus conclude that $P \covered_\motz Q$.
\end{proof}

Let $\dyck_{2n}(\nu)$ be the set of Dyck paths of type $\nu$ (which is a word in $N, E$). It is known in \cite{PRV2014extension} that, for any $\nu$, the Tamari lattice restricted to $\dyck_{2n}(\nu)$ is an interval. We denote this restriction by $(\leq_D, \dyck_{2n}(\nu))$. We have the following corollary on the structure of $(\leq_\motz, \motz_n)$.

\begin{coro} \label{coro:phi-cover}
  The poset $(\leq_\motz, \motz_n)$ is isomorphic to the union of intervals $(\leq_D, \dyck_{2n}(\nu))$ with all possible $\nu$ that avoids $EE$. The isomorphism is given by $\phi^{-1}$. Furthermore, each connected component in $(\leq_\motz, \motz_n)$ contains exactly all the paths in $\motz_n$ of a certain class.
\end{coro}
\begin{proof}
  It is clear that a Dyck path avoids $NNN$ if and only if its type avoids $EE$. The first point thus follows from Theorem~\ref{thm:phi-cover} and the fact that $\phi$ is a bijection. For the size of paths, given a Motzkin path $P$ of length $n$ with $k$ steps $N$, $k$ steps $E$ and $\ell$ steps $D$, we have $n = 2k+\ell$. Then, since $\phi^{-1}$ sends $N$ to $NNE$, $E$ to $E$ and $D$ to $NE$, the length of $\phi^{-1}(P)$ is $4k+2\ell = 2n$.

  For the second point, we deduce from Proposition~\ref{prop:phi-type} that Dyck paths of the same type correspond exactly to Motzkin paths of the same class. Since $(\leq_D, \dyck_{2n}(\nu))$ is an interval of the Tamari lattice, it is a connected poset. By Proposition~\ref{prop:comp-type}, paths of different classes are not comparable. We thus conclude that connected components in $(\leq_\motz, \motz_n)$ are in one-to-one correspondence with classes.
\end{proof}

We now know that classes of Motzkin paths characterize connected components in $(\leq_\motz, \motz_n)$. We also note that it was proved in \cite{PRV2014extension} that $(\leq_D, \dyck_{2n}(\nu))$ is isomorphic to the generalized Tamari lattice $\textsc{Tam}(\nu)$ defined therein.

\section{Enumerative aspect}

We now explore enumeration problems for $(\leq_\motz, \motz_n)$ with the structural results in the previous section. We have two major targets: the number of connected components and the number of intervals. The first one is easy.

\begin{prop}
  The number of connected components in $(\leq_\motz, \motz_n)$ is given by the $n$-th Fibonacci number $F_n$, defined by $F_1 = 1, F_2 = 1, F_{n+1} = F_{n} + F_{n-1}$.
\end{prop}
\begin{proof}
  By Corollary~\ref{coro:phi-cover}, connected components of $(\leq_\motz, \motz_n)$ are in bijection with words of length $n-1$ in $\{N,E\}$ avoiding $EE$, which are counted by Fibonacci numbers.
\end{proof}

It is not difficult to refine this result with respect to the number of diagonal steps.

\begin{prop} \label{prop:conn-motz-refined}
  The number of connected components in $(\leq_\motz, \motz_n)$ with $n-2k$ diagonal steps (thus $k$ north steps and $k$ east steps) in its elements is $\binom{n-k}{k}$.
\end{prop}
\begin{proof}
  By Corollary~\ref{coro:phi-cover} and the definition of $\phi$, the connected components to be counted are in bijection with words of length $n-1$ in $\{N,E\}$ avoiding $EE$ with $k$ occurrences of $E$. This number is given by $\binom{n-k}{k}$.
\end{proof}

To count intervals in $(\leq_\motz, \motz_n)$, by Corollary~\ref{coro:phi-cover}, we only need to count synchronized intervals avoiding $NNN$. We resort to the following known decomposition of synchronized intervals in \cite{nonsep}, with the reformulation in \cite{trinity}. A \tdef{properly pointed synchronized interval}, denoted by $[R^\ell \cdot R^r, S]$, is a synchronized interval with a split $R^\ell \cdot R^r$ in its lower path such that both $R^\ell$ and $R^r$ are Dyck paths, and $R^\ell$ is non-empty. We recall that the empty path is denoted by $\epsilon$.

\begin{prop}[Proposition~3.1 in \cite{trinity}] \label{prop:sync-decomp}
  Let $[R,S]$ be a synchronized interval in $(\leq_D, \dyck_{2n})$. The two Dyck paths $R$ and $S$ are uniquely decomposed as follows:
  \[
  R = N \cdot R_1^\ell \cdot E \cdot R_1^r \cdot R_2, \quad S = N \cdot S_1 \cdot E \cdot S_2.
  \]
  Here, the sub-paths $R_1^\ell, R_1^r, R_2, S_1, S_2$ satisfy
  \begin{itemize}
  \item Each sub-path is either empty or a Dyck path;
  \item $R_1^\ell = \epsilon$ if and only if $S_1 = \epsilon$, and in that case we also have $R_1^r = \epsilon$;
  \item $R_2 = \epsilon$ if and only if $S_2=\epsilon$;
  \item When not empty, $[R_1^\ell \cdot R_1^r, S_1]$ is a properly pointed synchronized interval, and $[R_2, S_2]$ is a synchronized interval.
  \end{itemize}
\end{prop}

We then have the following refined decomposition on synchronized intervals avoiding $NNN$.

\begin{prop} \label{prop:decomp-restrict}
  Let $[R,S]$ be a synchronized interval, with $S = N \cdot S_1 \cdot E \cdot S_2$ in the decomposition of Proposition~\ref{prop:sync-decomp}. Then, $S$ avoids $NNN$ if and only if both $S_1$ and $S_2$ are either empty or avoiding $NNN$, and if $S_1$ is not empty, then $S_1$ starts with $NE$.
\end{prop}
\begin{proof}
  Since $S_1$ and $S_2$ are separated by an east step $E$ in $S$, a pattern $NNN$ in $S$ occurs either in $S_1$, or in $S_2$, or at the beginning of $S$ when $S_1$ starts with $NN$. We thus have the equivalence. 
\end{proof}

We now use generating functions to enumerate synchronized intervals avoiding $EE$, following \cite{nonsep}. Given a Dyck path $R$, a \tdef{contact} is an intersection of $R$ with the main diagonal $x=y$. We denote by $\cont(R)$ the number of contacts of $R$. Since ultimately we want to count intervals of Motzkin paths, we will also track another statistic. Given a Dyck path $S$ avoiding $NNN$, we denote by $\ds(S)$ the number of north steps in $S$ that is neither followed nor preceded by another north steps. In other words, for $S$ written as $N^{a_1} E^{b_1} \cdots N^{a_k} E^{b_k}$ with $a_i \in \{1,2\}$ and $b_i>0$, the value of $\ds(S)$ is the number of $a_i$'s of value $1$. We have the following properties of $\ds$.

\begin{prop} \label{prop:ds}
  \begin{enumerate}
  \item Given a Dyck path $R$ avoiding $NNN$, its corresponding Motzkin path $\phi(R)$ has $\ds(R)$ diagonal steps.
  \item Given a synchronized interval $[R, S]$ avoiding $NNN$, we have $\ds(R) = \ds(S)$.
  \end{enumerate}
\end{prop}
\begin{proof}
  The first point comes from the definitions of $\phi$ and $\ds$. The second one is a direct consequence of Proposition~\ref{prop:phi-type}.
\end{proof}

We define the following generating function for intervals in $(\leq_D, \dyck^\circ_{2n})$ for all $n$:
\[
F_\circ(t,u,x) = \sum_{n > 0} t^n \sum_{R, S \in \dyck^\circ_{2n}, R \leq_D S} u^{\ds(S)} x^{\cont(R) - 1}.
\]

We then have the following functional equation for $F_\circ$.

\begin{prop} \label{prop:fn-eq}
  The generating function $F_\circ$ satisfies the following equation:
  \begin{align*}
  F_\circ(t,u,x) &= tux + tuxF_\circ(t,u,x) \\
  &+ t^2 x \frac{x(1+F_\circ(t,u,x)) - 1 - F_\circ(t,u,1)}{x-1} (1 + F_\circ(t,u,x)).
  \end{align*}
\end{prop}
\begin{proof}
  Each term corresponds to a case in the decomposition of a synchronized interval $[R,S]$ avoiding $NNN$, as in Proposition~\ref{prop:sync-decomp}, with restrictions in Proposition~\ref{prop:decomp-restrict}. The first term corresponds to the case $S_1 = S_2 = \epsilon$, where $R = S = NE$. The second term corresponds to the case $S_1 = \epsilon$ but $S_2 \neq \epsilon$, where $R = NE \cdot R_2$, adding one contact and increasing $\ds$ by $1$. The third term corresponds to the case $S_1 \neq \epsilon$, which is more complicated.

  First we observe that, if $S_1$ starts with $NE$, then the (not yet pointed) synchronized interval $[R_1, S_1]$ with $R_1 = R_1^\ell \cdot R_1^r$ can be written as $[NE \cdot R_1', NE \cdot S_1']$, with $[R_1', S_1']$ a synchronized interval avoiding $NNN$. The generating function of synchronized intervals avoiding $NNN$ that start with $NE$ is thus $tux(1+F_\circ(t,u,x))$. Then we notice that the extra $u$, contributed by $NE$, will not stand in the final interval, since the $NE$ at the beginning of $S_1$ becomes $NNE$ at the beginning of $S$, no longer contributing to $\ds$. Therefore, the final contribution should be $tx(1+F_\circ(t,u,x))$. Now, we see that $[R_1, S_1]$ gives exactly $\cont(R_1) - 1$ properly pointed synchronized intervals $[R_1^\ell \cdot R_1^r, S]$ with $\cont(R_1^\ell)$ ranging from $2$ to $\cont(R_1)$. This is because we can break $R_1$ at any of its contacts, except the first one, to give a properly pointed variant. In terms of generating function, the contribution $x^k t^n$ of $[R_1,S_1]$ turns into $t^n (x + x^2 + \ldots + x^k) = t^{n} x \frac{x^k-1}{x-1}$ for its properly pointed variants. Therefore, the contribution of $[R_1^\ell \cdot R_1^r, S_1]$ is
  \[
  x \frac{tx(1+F_\circ(t,u,x)) - t(1+F_\circ(t,u,1))}{x-1}.
  \]
  The contribution from $[R_2, S_2]$ is simply $1 + F_\circ$, and we also have two extra steps. We thus conclude this case.
\end{proof}

The equation in Proposition~\ref{prop:fn-eq} can be rearranged into:
\begin{equation} \label{eq:fn-eq}
  F_\circ(t,u,x) = tx(1 + F_\circ(t,u,x)) \left(u + t + t \cdot \frac{xF_\circ(t,u,x) - F_\circ(t,u,1)}{x-1} \right).
\end{equation}

Maybe not much of a surprise, (\ref{eq:fn-eq}) is very close to the functional equation of synchronized intervals in \cite{nonsep}. In particular, it is also a functional equation with one catalytic variable, in the scope of \cite{BMJ}. We thus know immediately from \cite{BMJ} without solving the equation that the generating function $F_\circ(t,u,x)$ is algebraic in its variables. We now solve (\ref{eq:fn-eq}) with the method in \cite{BMJ}. To simplify the notations, we denote $F_\circ \equiv F_\circ(t,u,x)$ and $F_1 \equiv F_1(t,u) \equiv F_\circ(t,u,1)$. We only need $F_1$ to be able to count intervals in $(\leq_\dyck, \dyck^\circ_{2n})$, which correspond to intervals in $(\leq_\motz, \motz_n)$. According to Proposition~\ref{prop:ds}, the variable $u$ counts the number of diagonal steps in elements of an interval in $(\leq_\motz, \motz_n)$.

\begin{thm} \label{thm:motz-int-cnt}
  The generating function $F_1(t,u)$ of intervals in $(\leq_\motz, \motz_n)$ is algebraic. More precisely, let $X$ be the formal power series in $t$ with coefficients polynomial in $u$ that satisfies the equation
\begin{equation} \label{eq:X}
  u^2 t^2 X^5 - t^2 (1 + u^2) X^4 - 2 u t X^3 + 2 u t X^2 + X - 1 = 0.
\end{equation}
Then the series $F_1$ can be expressed in terms of $X$ as
\begin{equation} \label{eq:f1}
  F_1(t,u) = \frac{u^2 t^2 X^4 - t (u^2 t + u t^2 + u + 2t) X^3 + (1 + u t + t^2) X - 1}{t^2 X (u t X^2 - 1)}.
\end{equation}
\end{thm}
\begin{proof}
A rearrangement of (\ref{eq:fn-eq}) gives
\begin{align}
  \begin{split}\label{eq:sys-1}
    t^2 x^2 F_\circ^2 + (2 t^2 x^2 + x^2 u t - x u t &- t^2 x - x + 1 - t^2 x F_1) F_\circ \\ &+ t^2 x^2 - t^2 x + x^2 u t - x u t - t^2 x F_1 = 0.
  \end{split}
\end{align}
We notice that $F_1$ does not depend on $x$. We now regard the left-hand side of (\ref{eq:sys-1}) as a polynomial $P(F_\circ, F_1, t, u, x)$. Differentiating (\ref{eq:sys-1}) by $x$, we have
\[
\left(\frac{\partial F_\circ}{\partial x}\right) \cdot \frac{\partial P}{\partial F_\circ}(F_\circ, F_1, t, u, x) + \frac{\partial P}{\partial x}(F_\circ, F_1, t, u, x) = 0.
\]
If there is some Puiseux series $X$ such that the substitution $\frac{\partial P}{\partial F_\circ}(F_\circ, F_1, t, u, X) = 0$, then automatically we have $\frac{\partial P}{\partial x}(F_\circ, F_1, t, u, X) = 0$ after substitution. A simple computation of the partial differentiations gives the following equations:
\begin{equation} \label{eq:sys-2}
  2 t^2 X^2 F_\circ(X) + (2 t^2 X^2 + X^2 u t - X u t - t^2 X - X + 1 - t^2 X F_1) = 0,
\end{equation}
\begin{align}
  \begin{split}\label{eq:sys-3}
    2 t^2 X F_\circ^2(X) + (4 t^2 X + 2 X u t - u t &- t^2 - 1 - t^2 F_1) F_\circ(X) \\ &+ 2 t^2 X - t^2 + 2 X u t - u t - t^2 F_1 = 0.
  \end{split}
\end{align}
Along with (\ref{eq:sys-1}) with $x$ substituted by $X$, we have a system of three polynomial equations with three unknowns $F_\circ(X), X, F_1$. To see that there is only one power series $X$ in $t$ that satisfies (\ref{eq:X}), we observe that (\ref{eq:X}) can be written as $X = 1 + tQ(X)$, where $Q(X)$ is a polynomial in $X$ with coefficients polynomial in $u,t$. Therefore, we have $X=1+O(t)$, and its coefficients can be computed iteratively, thus determined, and they are clearly polynomials in $u$. After picking the unique $X$, we can thus solve for $F_1$ (preferably with a computer algebra system), which gives the announced result. 
\end{proof}

By substituting (\ref{eq:f1}) into (\ref{eq:fn-eq}), we can solve for $F_\circ$, which means $F_\circ \equiv F_\circ(t,u,x)$ is also an algebraic series in $t, u, x$. We omit the exact expression here.

The first terms of $F_1(t, 1)$, whose coefficient of $t^n$ is the number of intervals in $(\leq_M, \motz_n)$ thanks to Corollary~\ref{coro:phi-cover}, are
\[
F_1(t, 1) = t + 2t^2 + 5t^3 + 14t^4 + 43t^5 + 140t^6 + 477t^7 + 1638t^8 + 6106t^9 + \cdots.
\]
These values agree with experimental results. The sequence
\[
1, 2, 5, 14, 43, 140, 477, 1638, \ldots
\]
has appeared on OEIS as \OEIS{A307787}, which counts the number of valid hook configurations of $132$-avoiding permutations (\textit{cf.} \cite{valid-hook}). 

\section{Extension to Schr\"oder paths}

All our constructions and results can be transferred to Schr\"oder paths, which are essentially Motzkin paths where diagonal steps are counted as of length $2$. It is thus clear that every Schr\"oder path is of even length. We denote by $\schr_{2n}$ the set of Schr\"oder paths of length $2n$. We can construct a partial order $(\leq_\schr, \schr_{2n})$ in the same way as $(\leq_\motz, \motz_n)$, since the only difference is how we count the length of a path.

By Corollary~\ref{coro:phi-cover}, elements in each connected component of $(\leq_\motz, \motz_n)$ have the same number of diagonal steps, hence are Schr\"oder paths of the same length. Therefore, the partial order $(\leq_\schr, \schr_{2n})$ is also isomorphic to a disjoint union of $(\leq_\dyck, \dyck(\nu))$ with appropriate $\nu$'s. We can thus deduce the following enumeration results for $(\leq_\schr, \schr_{2n})$.

\begin{prop}
  There are $2^n$ connected components in $(\leq_\schr, \schr_{2n})$.
\end{prop}
\begin{proof}
  We observe that Schr\"oder paths in $\schr_{2n}$ are exactly those ending at $(n,n)$. On the path, when passing from $y$-coordinate $k$ to $k+1$, we have either an east step or a diagonal steps. We thus have $2^n$ different classes in $\schr_{2n}$, each corresponding to a connected component according to Corollary~\ref{coro:phi-cover}. This result can also be seen as a consequence of Proposition~\ref{prop:conn-motz-refined}, with a translation between the Schr\"oder and the Motzkin path length.
\end{proof}

Let $G(t,u)$ be the generating function for intervals in $(\leq_\schr, \schr_{2n})$ for all $n$ defined as
\[
G(t,u) = \sum_{n > 0} t^n \sum_{P, Q \in \schr_{2n}, P \leq_\schr Q} u^{\#\; \mathrm{diagonal\;steps\;in}\;P}.
\]
We can now deduce $G(t,u)$ from $F_1(t,u)$.
\begin{thm} \label{thm:schr-int-cnt}
  The generating function $G(t,u)$ of intervals in $(\leq_\schr, \schr_n)$ is algebraic. More precisely, let $X'$ is the formal power series in $t$ with coefficients polynomial in $u$ that satisfies the equation
\begin{equation} \label{eq:X-s}
  u^2 t^2 X'^5 - t (1 + u^2 t) X'^4 - 2 u t X'^3 + 2 u t X'^2 + X' - 1 = 0.
\end{equation}
Then $G(t,u)$ can be expressed in terms of $X'$ as
\begin{equation} \label{eq:f1-s}
  G(t,u) = \frac{u^2 t^2 X'^4 - (u^2 t^2 + u t^2 + u t + 2t) X'^3 + (1 + u t + t) X' - 1}{t X' (u t X'^2 - 1)}.
\end{equation}
\end{thm}
\begin{proof}
  Since diagonal steps are counted as $2$ towards the length of a Schr\"oder path, we have
  \[
  G(t^2, u) = F_1(t, ut).
  \]
  The result follows from appropriate substitutions of formulas in Theorem~\ref{thm:motz-int-cnt}.
\end{proof}

The first terms of $G(t,1)$, whose coefficient of $t^n$ is the number of intervals in $(\leq_\schr, \schr_{2n})$, are
\[
G(t,1) = 2t + 8t^2 + 46t^3 + 320t^4 + 2500t^5 + 21120t^6 +188758t^7 + 1760256t^8 + \cdots.
\]
The sequence of its coefficients is not yet on OEIS.

\section{Discussions}

We now discuss the result of Baril and Pallo in \cite{MR3128387}. It has a flavor that is very close to our result. More precisely, they analyzed the sub-poset of the Tamari lattice induced by the so-called ``Motzkin words'', which are well-parenthesized words defined by the generative grammar $S \rightsquigarrow \epsilon \mid (SS)$, with $\epsilon$ the empty word. If we read opening (resp. closing) parenthesis as north (resp. east) steps, the set of Motzkin words can be regarded as a set of Dyck paths, with the generative grammar $R \rightsquigarrow \epsilon \mid N \cdot R \cdot R \cdot E$. We can prove by a simple induction that every Dyck path generated in this way, which corresponds to a Motzkin word, must have the form $N \cdot R' \cdot E$, with $R'$ a Dyck path whose type avoids $NN$. However, we know from Proposition~5.2 and Theorem~1.2 of \cite{PRV2014extension} that Dyck paths with type $w$ are in bijection with those of type $\overleftarrow{w}$, where $\overleftarrow{w}$ is the word $w$ read from right to left while replacing $N$ by $E$ and $E$ by $N$. Furthermore, this bijection is an order isomorphism from the Tamari lattice to its order dual. As a consequence, under our definition of the Tamari lattice in Section~2, the restriction of the Tamari lattice to Dyck paths whose types avoid $NN$ is isomorphic to the order dual of the restriction to Dyck paths whose types avoid $EE$, which is exactly the poset we studied in Section~3, isomorphic to our poset of Motzkin paths.

Then why did Baril and Pallo had a different poset (which is connected by a maximal element) from ours, if we studied the same poset restricted to (roughly) the same set of elements? It is because our definitions of the Tamari lattice on Dyck paths differ. More precisely, both the definition here and that in \cite{MR3128387} can be seen as coming from the Tamari lattice defined on binary trees, where the order relation is given by tree rotation (\textit{cf}. \cite{PRV2014extension}). We then have different ways to convert binary trees into Dyck paths. Given a binary tree $T$, either it is empty, denoted by $T = \epsilon_T$, or it has the form $T = (T_\ell, T_r)$, where $T_\ell$ (resp. $T_r$) is the left (resp. right) sub-tree. There are at least two ways to define a bijection from binary trees to Dyck paths recursively. The first one is what we take here implicitly, which is also taken in various other works \cite{PRV2014extension,BB2009intervals,nonsep}:
\[ \delta_1(\epsilon_T) = \epsilon, \quad \delta_1((T_\ell, T_r)) = \delta_1(T_\ell) \cdot N \cdot \delta_1(T_r) \cdot E. \]
Another is the one taken in \cite{MR3128387}:
\[ \delta_2(\epsilon_T) = \epsilon, \quad \delta_2((T_\ell, T_r)) = N \cdot \delta_2(T_\ell) \cdot E \cdot \delta_2(T_r). \]
Since the mappings are different, it is reasonable that the posets obtained are different, as the same tree is mapped to different Dyck paths. Via the two mappings, we are in fact looking at different portions of the Tamari lattice, leading to different posets.

Motivated by the generalization from the Tamari lattice to the $m$-Tamari lattice, we can also consider similar constructions defined on $m$-ballot paths, a generalization of Dyck paths. An \emph{$m$-ballot path} is a lattice path formed by north steps and east steps that always stays above the $m$-diagonal $x = my$. The construction in Section~\ref{sec:prelim} that defines partial orders on Dyck path and Motzkin paths, when applied to $m$-ballot paths, gives the $m$-Tamari lattice (see \cite{bergeron-preville}). For the counterpart of Motzkin paths in this case, there are two natural choices for the ``diagonal step'': either we take the usual diagonal step $D = (1,1)$, or we take the $m$-diagonal step $D_m = (m,1)$. For both cases, the structure of the poset is not clear and requires further exploration.

\section*{Acknowledgement}

The author thanks Cyril Banderier for raising the question on a possible generalization of the Tamari lattice on Schr\"oder paths. The discussion took place during the Workshop ``Enumerative Combinatorics'' at Erwin Schr\"odinger Institute.

\bibliographystyle{alpha}
\bibliography{motzkin}

\begin{thebibliography}{Fan18b}

\bibitem[BB09]{BB2009intervals}
O.~Bernardi and N.~Bonichon.
\newblock Intervals in {C}atalan lattices and realizers of triangulations.
\newblock {\em J. Combin. Theory Ser. A}, 116(1):55--75, 2009.

\bibitem[BMJ06]{BMJ}
M.~Bousquet-M{\'e}lou and A.~Jehanne.
\newblock Polynomial equations with one catalytic variable, algebraic series
  and map enumeration.
\newblock {\em J. Combin. Theory Ser. B}, 96(5):623--672, 2006.

\bibitem[BP14]{MR3128387}
J.-L. Baril and J.-M. Pallo.
\newblock Motzkin subposets and {M}otzkin geodesics in {T}amari lattices.
\newblock {\em Inform. Process. Lett.}, 114(1-2):31--37, 2014.

\bibitem[BPR12]{bergeron-preville}
F.~Bergeron and L.-F. Pr{\'e}ville-Ratelle.
\newblock Higher trivariate diagonal harmonics via generalized {T}amari posets.
\newblock {\em J. Comb.}, 3(3):317--341, 2012.

\bibitem[Cal04]{callan}
D.~Callan.
\newblock Two bijections for {D}yck path parameters.
\newblock arXiv:math/0406381 [math.CO], 2004.

\bibitem[Cha05]{chapoton-tamari}
F.~Chapoton.
\newblock Sur le nombre d'intervalles dans les treillis de {T}amari.
\newblock {\em S\'em. Lothar. Combin.}, 55:Art. B55f, 18 pp. (electronic),
  2005.

\bibitem[CP13]{interval-poset}
G.~Ch{\^a}tel and V.~Pons.
\newblock Counting smaller trees in the {T}amari order.
\newblock In {\em 25th {I}nternational {C}onference on {F}ormal {P}ower
  {S}eries and {A}lgebraic {C}ombinatorics ({FPSAC} 2013)}, Discrete Math.
  Theor. Comput. Sci. Proc., AS, pages 433--444. Assoc. Discrete Math. Theor.
  Comput. Sci., Nancy, 2013.

\bibitem[Fan18a]{sticky}
W.~Fang.
\newblock Planar triangulations, bridgeless planar maps and tamari intervals.
\newblock {\em European J. Combin.}, 70:75--91, 2018.

\bibitem[Fan18b]{trinity}
W.~Fang.
\newblock A trinity of duality: non-separable planar maps, $\beta$-(0,1) trees
  and synchronized intervals.
\newblock {\em Adv. Appl. Math.}, 95:1--30, 2018.

\bibitem[FP05]{MR2202340}
L.~Ferrari and R.~Pinzani.
\newblock Lattices of lattice paths.
\newblock {\em J. Statist. Plann. Inference}, 135(1):77--92, 2005.

\bibitem[FPR17]{nonsep}
W.~Fang and L.-F. Pr{\'e}ville-Ratelle.
\newblock The enumeration of generalized {T}amari intervals.
\newblock {\em European J. Combin.}, 61:69--84, 2017.

\bibitem[LR98]{loday1998hopf}
J.-L. Loday and M.~O. Ronco.
\newblock Hopf algebra of the planar binary trees.
\newblock {\em Adv. Math.}, 139(2):293--309, 1998.

\bibitem[PRV17]{PRV2014extension}
L.-F. Pr{\'e}ville-Ratelle and X.~Viennot.
\newblock The enumeration of generalized {T}amari intervals.
\newblock {\em Trans. Amer. Math. Soc.}, 369(7):5219--5239, July 2017.
\newblock arXiv:1406.3787.

\bibitem[Sta63]{associahedron}
J.~D. Stasheff.
\newblock Homotopy associativity of {$H$}-spaces. {I}, {II}.
\newblock {\em Trans. Amer. Math. Soc. 108 (1963), 275-292; ibid.},
  108:293--312, 1963.

\bibitem[Tam62]{tamari}
D.~Tamari.
\newblock The algebra of bracketings and their enumeration.
\newblock {\em Nieuw Arch. Wisk. (3)}, 10:131--146, 1962.

\end{thebibliography}

\end{document}